\documentclass[12pt, reqno]{amsart}
\usepackage{amsmath, amsthm, amscd, amsfonts, amssymb, graphicx, color}
\usepackage[bookmarksnumbered, colorlinks, plainpages]{hyperref}

\newtheorem{theorem}{Theorem}[section]
\newtheorem{lemma}[theorem]{Lemma}

\newtheorem{corollary}[theorem]{Corollary}
\theoremstyle{definition}
\newtheorem{definition}[theorem]{Definition}

\theoremstyle{remark}

\numberwithin{equation}{section}
\input{tcilatex}

\begin{document}
\title[Generalized Struve Functions]{Some Properties of a Class of Analytic
Functions defined by Generalized Struve Functions }
\author[M. Raza ]{Mohsan Raza$^{a}$}
\address{ $^{a}$Department of Mathematics, GC University Faisalabad,
Pakistan.}
\email{\textcolor[rgb]{0.00,0.00,0.84}{mohsan976@yahoo.com}}
\author[N. Ya\u{g}mur ]{Nihat YA\u{G}MUR$^{b}$}
\address{$^{b}$Department of Mathematics, Faculty of Science and Art,
Erzincan University, Erzincan, 24000, Turkey.}
\email{nhtyagmur@gmail.com}
\date{Received:\\
\indent$^{\ast }$ Corresponding author\\
2010\textit{\ Mathematics Subject Classification. }30C45, 30C80, 33C10.}

\begin{abstract}
The aim of this paper is to define \ a new operator by using the generalized
Struve functions $\tsum_{n=0}^{\infty }\frac{\left( -c/4\right) ^{n}}{\left(
3/2\right) _{n}\left( k\right) _{n}}z^{n+1}\ $with $k$ $=p+$ $\left(
b+2\right) /2\neq 0,-1,-2,\ldots $ and $b,c,k\in 
%TCIMACRO{\U{2102} }%
%BeginExpansion
\mathbb{C}
%EndExpansion
$. By using this operator we define a subclass of analytic functions. We
discuss some properties of this class such as inclusion problems, radius
problems and some other interesting properties related with this operator.
\end{abstract}

\maketitle

\section{Introduction}

Let $A$ be the class of functions $f$ of the form%
\begin{equation}
f\left( z\right) =z+\overset{\infty }{\underset{n=2}{\sum }}a_{n}z^{n},
\label{a}
\end{equation}%
which are analytic in the open unit disk $E=\left\{ z:\left\vert
z\right\vert <1\right\} $. A function $f$ is said to be subordinate to a
function $\ g$ written as $f\prec g$, if there exists a Schwarz function $w$
with $w\left( 0\right) =0$ and $\left\vert w\left( z\right) \right\vert <1$
such that $f\left( z\right) =g\left( w\left( z\right) \right) .$ In
particular if $g$ is univalent in $E,\ $then $f\left( 0\right) =g\left(
0\right) $ and $f\left( E\right) \subset g\left( E\right) .$

For any two analytic functions$\ f\left( z\right) \ $and$\ g\left( z\right)
\ $with%
\begin{equation*}
f\left( z\right) =\overset{\infty }{\underset{n=0}{\dsum }}b_{n}z^{n+1}\ 
\text{and}\ g\left( z\right) =\overset{\infty }{\underset{n=0}{\dsum }}%
c_{n}z^{n+1},\ z\in E,
\end{equation*}%
the convolution $\left( \text{Hadamard product}\right) \ $is given by%
\begin{equation*}
\left( f\ast g\right) \left( z\right) =\overset{\infty }{\underset{n=0}{%
\dsum }}b_{n}c_{n}z^{n+1},\ z\in E.
\end{equation*}

Consider the second order inhomogeneous differential equation, for some
details\ see \cite{Zang},%
\begin{equation}
z^{2}w^{\prime \prime }\left( z\right) +zw^{\prime }\left( z\right) +\left(
z^{2}-p^{2}\right) w\left( z\right) =\frac{4\left( z/2\right) ^{p+1}}{\sqrt{%
\pi }\Gamma \left( p+1/2\right) }.  \label{b}
\end{equation}%
The solution of the homogeneous part is Bessel functions of order $p$, where 
$p$ is real or complex number. The particular solution of the inhomogenious
equation defined in $\left( \ref{b}\right) $ is called the struve function
of order $p.$ It is defined as 
\begin{equation}
H_{p}\left( z\right) =\overset{\infty }{\underset{n=0}{\dsum }}\frac{\left(
-1\right) ^{n}\left( z/2\right) ^{2n+p+1}}{\Gamma \left( n+3/2\right) \Gamma
\left( p+n+3/2\right) }.  \label{c}
\end{equation}%
Now we consider the differential equation%
\begin{equation}
z^{2}w^{\prime \prime }\left( z\right) +zw^{\prime }\left( z\right) -\left(
z^{2}+p^{2}\right) w\left( z\right) =\frac{4\left( z/2\right) ^{p+1}}{\sqrt{%
\pi }\Gamma \left( p+1/2\right) }  \label{c1}
\end{equation}%
The equation $\left( \ref{c1}\right) $differs from the equation $\left( \ref%
{b}\right) $ in the coefficients of $w\left( z\right) $. Its particular
solution is called the modified struve functions of order $p$ and is given as%
\begin{equation*}
L_{p}\left( z\right) =-ie^{-ip\pi /2}H_{p}\left( iz\right) =\overset{\infty }%
{\underset{n=0}{\dsum }}\frac{\left( z/2\right) ^{2n+p+1}}{\Gamma \left(
n+3/2\right) \Gamma \left( p+n+3/2\right) }.
\end{equation*}%
Again consider the second order inhomogenous differential equation%
\begin{equation}
z^{2}w^{\prime \prime }\left( z\right) +bzw^{\prime }\left( z\right) +\left[
cz^{2}-p^{2}+\left( 1-b\right) p\right] w\left( z\right) =\frac{4\left(
z/2\right) ^{p+1}}{\sqrt{\pi }\Gamma \left( p+b/2\right) },  \label{d}
\end{equation}%
where $b,c,p\in 
%TCIMACRO{\U{2102} }%
%BeginExpansion
\mathbb{C}
%EndExpansion
.$ The equation $\left( \ref{d}\right) $ generalizes the equation $\left( %
\ref{b}\right) $ and $\left( \ref{c1}\right) .$ In particular for $b=1,\
c=1, $ we obtain $\left( \ref{b}\right) $ and for $b=1,\ c=-1,$ we obtain $%
\left( \ref{c1}\right) .$ Its particular solution has the series form%
\begin{equation}
M_{p,b,c}\left( z\right) =\overset{\infty }{\underset{n=0}{\dsum }}\frac{%
\left( -1\right) ^{n}c^{n}\left( z/2\right) ^{2n+p+1}}{\Gamma \left(
n+3/2\right) \Gamma \left( p+n+\left( b+2\right) /2\right) }.  \label{e}
\end{equation}%
and is called the generalized struve function of order$\ p.\ $This series is
convergent every where but not univalent in the open unit disk $E.$ We take
the transformation%
\begin{equation}
N_{p,b,c}\left( z\right) =2^{p}\sqrt{\pi }\Gamma \left( p+\left( b+2\right)
/2\right) z^{\left( -p-1\right) /2}M_{p,b,c}\left( \sqrt{z}\right) =\overset{%
\infty }{\underset{n=0}{\dsum }}\frac{\left( -c/4\right) ^{n}z^{n}}{\left(
3/2\right) _{n}\left( k\right) _{n}},  \label{f}
\end{equation}%
where $k=p+\left( b+2\right) /2\neq 0,-1,-2,\ldots $ and $\left( \gamma
\right) _{n}=\frac{\Gamma \left( \gamma +n\right) }{\Gamma \left( \gamma
\right) }=\gamma \left( \gamma +1\right) \ldots \left( \gamma +n-1\right) .$
This function is analytic in the whole complex plane and satisfies the
differential equation%
\begin{equation*}
4z^{2}w^{\prime \prime }\left( z\right) +2\left( 2p+b+3\right) zw^{\prime
}\left( z\right) +\left[ cz+2p+b\right] w\left( z\right) =2p+b.
\end{equation*}%
Some geometric properties such as univalency, starlikeness, convexity,
close-to-convexity of the function $N_{p,b,c}\left( z\right) $ has been
studied recently by Orhan and Yagmur \cite{orh} and Yagmur and Orhan \cite%
{orh1,yag}.

Dziok and Srivastava \cite{dziok1,dziok} defined the linear operator $H$ by
using the generalized hypergeometric functions and is given as $H\left(
\alpha _{1},\ldots \alpha _{s};\beta _{1},\ldots \beta _{q}\right)
:A\rightarrow A$ with $\alpha _{i}\in 
%TCIMACRO{\U{2102} }%
%BeginExpansion
\mathbb{C}
%EndExpansion
\left( i=1,2,\ldots ,s\right) $ and $\beta _{i}\in 
%TCIMACRO{\U{2102} }%
%BeginExpansion
\mathbb{C}
%EndExpansion
\backslash 
%TCIMACRO{\U{2124} }%
%BeginExpansion
\mathbb{Z}
%EndExpansion
_{0}^{-}\left( i=1,2,\ldots ,q\right) $ such that

\begin{equation*}
H\left( \alpha _{1},\ldots \alpha _{s};\beta _{1},\ldots \beta _{q}\right)
f\left( z\right) =z_{s}F_{q}\left( \alpha _{1},\ldots \alpha _{s};\beta
_{1},\ldots \beta _{q};z\right) \ast f\left( z\right) ,
\end{equation*}%
where%
\begin{equation*}
_{s}F_{q}\left( \alpha _{1},\ldots \alpha _{s};\beta _{1},\ldots \beta
_{q};z\right) =\overset{\infty }{\underset{n=0}{\dsum }}\frac{\left( \alpha
_{1}\right) _{n}\ldots \left( \alpha _{s}\right) _{n}z^{n}}{\left( \beta
_{1}\right) _{n}\ldots \left( \beta _{q}\right) _{n}n!},\ \ s\leq q+1;s,q\in 
%TCIMACRO{\U{2115} }%
%BeginExpansion
\mathbb{N}
%EndExpansion
_{0}=%
%TCIMACRO{\U{2115} }%
%BeginExpansion
\mathbb{N}
%EndExpansion
\cup \left\{ 0\right\}
\end{equation*}%
is the generalized hypergeometric function. Deniz \cite{deniz} use similar
argument to define a convolution operator $B_{k}^{c}:A\rightarrow A$ by
using generalized Bessel functions and is given as 
\begin{equation*}
B_{k}^{c}f\left( z\right) =\varphi _{k,c}\left( z\right) \ast f\left(
z\right) =z+\overset{\infty }{\underset{n=1}{\dsum }}\frac{\left(
-c/4\right) ^{n}a_{n+1}z^{n+1}}{\left( k\right) _{n}n!},\ \ \left( k=p+\frac{%
b+1}{2}\notin 
%TCIMACRO{\U{2124} }%
%BeginExpansion
\mathbb{Z}
%EndExpansion
_{0}^{-},c\in 
%TCIMACRO{\U{2102} }%
%BeginExpansion
\mathbb{C}
%EndExpansion
\right) ,
\end{equation*}%
where%
\begin{equation*}
\varphi _{k,c}\left( z\right) =z+\overset{\infty }{\underset{n=1}{\dsum }}%
\frac{\left( -c/4\right) ^{n}z^{n+1}}{\left( k\right) _{n}n!}.
\end{equation*}%
For some refrences for convolution operators see \cite{raina, sh, Atia}.

Now using $\left( \ref{f}\right) ,$ we define the following convolution
operator. Let%
\begin{equation*}
\varphi _{p,b,c}\left( z\right) =2^{p}\sqrt{\pi }\Gamma \left( p+\left(
b+2\right) /2\right) z^{\left( -p+1\right) /2}M_{p,b,c}\left( \sqrt{z}%
\right) =z+\overset{\infty }{\underset{n=1}{\dsum }}\frac{\left( -c/4\right)
^{n}z^{n+1}}{\left( 3/2\right) _{n}\left( k\right) _{n}}
\end{equation*}%
Then%
\begin{equation}
S_{k}^{c}f\left( z\right) =\varphi _{p,b,c}\left( z\right) \ast f\left(
z\right) =z+\overset{\infty }{\underset{n=1}{\dsum }}\frac{\left(
-c/4\right) ^{n}a_{n+1}z^{n+1}}{\left( 3/2\right) _{n}\left( k\right) _{n}}\
\ \left( k=p+\frac{b+2}{2}\notin 
%TCIMACRO{\U{2124} }%
%BeginExpansion
\mathbb{Z}
%EndExpansion
_{0}^{-},b,c,p\in 
%TCIMACRO{\U{2102} }%
%BeginExpansion
\mathbb{C}
%EndExpansion
\right) .  \label{g}
\end{equation}%
It can easily be seen that%
\begin{equation}
z\left( S_{k+1}^{c}f\left( z\right) \right) ^{\prime }=kS_{k}^{c}f\left(
z\right) -\left( k-1\right) S_{k+1}^{c}f\left( z\right) .  \label{h}
\end{equation}%
\textbf{Special cases}

(i) \ \ \ For $b=1,\ c=1,$ we have the operator $\mathcal{S}%
_{p}:A\rightarrow A$ related with struve function of order $p.$ It is given
as%
\begin{eqnarray*}
\mathcal{S}_{p}f\left( z\right) &=&\varphi _{p,1,1}\left( z\right) \ast
f\left( z\right) =\left[ 2^{p}\sqrt{\pi }\Gamma \left( p+3/2\right)
z^{\left( -p+1\right) /2}M_{p,1,1}\left( \sqrt{z}\right) \right] \ast
f\left( z\right) \\
&=&z+\overset{\infty }{\underset{n=1}{\dsum }}\frac{\left( -1/4\right)
^{n}a_{n+1}z^{n+1}}{\left( 3/2\right) _{n}\left( p+3/2\right) _{n}}
\end{eqnarray*}%
and the recursive relation 
\begin{equation*}
z\left[ \mathcal{S}_{p+1}f\left( z\right) \right] ^{\prime }=\left(
p+3/2\right) \mathcal{S}_{p}f\left( z\right) -\left( p+1/2\right) \mathcal{S}%
_{p+1}f\left( z\right)
\end{equation*}
holds.

(ii) \ \ \ For $b=1,\ c=-1,$ we obtain the operator $\mathfrak{S}%
_{p}:A\rightarrow A$ related with modified struve function of order $p.$ It
is given as%
\begin{eqnarray*}
\mathfrak{S}_{p}f\left( z\right) &=&\varphi _{p,1,-1}\left( z\right) \ast
f\left( z\right) =\left[ 2^{p}\sqrt{\pi }\Gamma \left( p+3/2\right)
z^{\left( -p+1\right) /2}M_{p,1,-1}\left( \sqrt{z}\right) \right] \ast
f\left( z\right) \\
&=&z+\overset{\infty }{\underset{n=1}{\dsum }}\frac{\left( 1/4\right)
^{n}a_{n+1}z^{n+1}}{\left( 3/2\right) _{n}\left( p+3/2\right) _{n}}
\end{eqnarray*}%
and the recursive relation%
\begin{equation*}
z\left[ \mathfrak{S}_{p+1}f\left( z\right) \right] ^{\prime }=\left(
p+3/2\right) \mathfrak{S}_{p}f\left( z\right) -\left( p+1/2\right) \mathfrak{%
S}_{p+1}f\left( z\right)
\end{equation*}%
holds.

We define the following class of analytic functions by using the operator $%
S_{k}^{c}f\left( z\right) .$

\begin{definition}
Let $f\in A$. Then $f\in N_{k,c}^{\alpha }\left( \lambda ,\mu ,\phi \right) $
for $0<\mu <1,\lambda \in 
%TCIMACRO{\U{2102} }%
%BeginExpansion
\mathbb{C}
%EndExpansion
,k=p+\left( b+2\right) /2\neq 0,-1,-2,\ldots ,b,c,p\in 
%TCIMACRO{\U{2102} }%
%BeginExpansion
\mathbb{C}
%EndExpansion
,$ and $\left\vert \alpha \right\vert <\frac{\pi }{2},$ if and only if 
\begin{equation}
e^{i\alpha }\left\{ \left( 1+\lambda \right) \left( \frac{z}{%
S_{k+1}^{c}f\left( z\right) }\right) ^{\mu }-\lambda \frac{S_{k}^{c}f\left(
z\right) }{S_{k+1}^{c}f\left( z\right) }\left( \frac{z}{S_{k+1}^{c}f\left(
z\right) }\right) ^{\mu }\right\} \prec \cos \alpha \phi \left( z\right)
+i\sin \alpha ,  \label{i}
\end{equation}%
where $\phi \left( z\right) $ is a convex univalent function with $\phi
\left( 0\right) =1.$
\end{definition}

(i) \ \ \ For $\phi \left( z\right) =\frac{1+Az}{1+Bz},-1\leq B<A\leq 1,$ we
have the class $N_{k,c}^{\alpha }\left( \lambda ,\mu ,\frac{1+Az}{1+Bz}%
\right) ,$ which

consists of functions $f$ such that%
\begin{equation*}
J\left( \alpha ,c,k,f\left( z\right) \right) \prec \frac{1+Az}{1+Bz},
\end{equation*}%
where%
\begin{eqnarray*}
&&J\left( \alpha ,c,k,f\left( z\right) \right) \\
&=&\frac{1}{\cos \alpha }\left[ e^{i\alpha }\left\{ \left( 1+\lambda \right)
\left( \frac{z}{S_{k+1}^{c}f\left( z\right) }\right) ^{\mu }-\lambda \frac{%
S_{k}^{c}f\left( z\right) }{S_{k+1}^{c}f\left( z\right) }\left( \frac{z}{%
S_{k+1}^{c}f\left( z\right) }\right) ^{\mu }\right\} -i\sin \alpha \right] .
\end{eqnarray*}%
(ii) \ \ For\ $\phi \left( z\right) =\frac{1+z}{1-z},$ we have the class $%
N_{k,c}^{\alpha }\left( \lambda ,\mu ,\frac{1+z}{1-z}\right) .$ That is $%
f\in $ $N_{k,c}^{\alpha }\left( \lambda ,\mu ,\frac{1+z}{1-z}\right) $ if%
\begin{equation*}
J\left( \alpha ,c,k,f\left( z\right) \right) \prec \frac{1+z}{1-z}.
\end{equation*}%
Since it is well known that for a function $p\left( z\right) \prec \frac{1+z%
}{1-z},$ then $\func{Re}p\left( z\right) >0.$ This implies that $f\in $ $%
N_{k,c}^{\alpha }\left( \lambda ,\mu ,\frac{1+z}{1-z}\right) $ if%
\begin{equation*}
\func{Re}J\left( \alpha ,c,k,f\left( z\right) \right) >0.
\end{equation*}

\begin{lemma}
\label{3}\cite{lio} Let $F\ $be analytic and convex in $E$. If $f,\ g\in A\ $%
and $f,\ g\prec F.\ $Then%
\begin{equation*}
\sigma f+\left( 1-\sigma \right) g\prec F,\ \ \ 0\leq \sigma \leq 1.
\end{equation*}
\end{lemma}

\begin{lemma}
\label{4}\cite{hallen} Let $h$ be convex in $E$ with $h(0)=a$ and $\beta \in 
%TCIMACRO{\U{2102} }%
%BeginExpansion
\mathbb{C}
%EndExpansion
$ such that $\func{Re}\beta \geq 0$. If $p\in H\left[ a,n\right] $ and%
\begin{equation*}
p(z)+\frac{zp%
%TCIMACRO{\U{b4}}%
%BeginExpansion
{\acute{}}%
%EndExpansion
(z)}{\beta }\prec h\left( z\right) ,
\end{equation*}%
then $p(z)\prec q\left( z\right) \prec h\left( z\right) ,$ where%
\begin{equation*}
q\left( z\right) =\frac{\beta }{nz^{\beta /n}}\dint\limits_{0}^{z}h\left(
t\right) t^{\beta /n-1}dt
\end{equation*}
\end{lemma}

and $q\left( z\right) $ is the best dominant.

\begin{lemma}
\label{1}\cite{abr}. Let $a,\ b\ $and\ $c\neq 0,-1,-2\ldots \ $be complex
numbers.\ Then, for $\func{Re}c>\func{Re}b>0$%
\begin{eqnarray*}
\left( i\right) \ \ \ _{2}F_{1}\left( a,b,c;z\right) &=&\frac{\Gamma \left(
c\right) }{\Gamma \left( c-b\right) \Gamma \left( b\right) }\underset{0}{%
\overset{1}{\dint }}t^{b-1}\left( 1-t\right) ^{c-b-1}\left( 1-tz\right)
^{-a}dt, \\
\left( ii\right) \ \ \ _{2}F_{1}\left( a,b,c;z\right) &=&\ _{2}F_{1}\left(
b,a,c;z\right) , \\
\left( iii\right) \ \ _{2}F_{1}\left( a,b,c;z\right) &=&\left( 1-z\right)
^{-a}\ _{2}F_{1}\left( a,c-b,c;\frac{z}{z-1}\right) .
\end{eqnarray*}
\end{lemma}

\begin{lemma}
\label{6}\cite{liu study} Let $-1\leq B_{1}\leq B_{2}<A_{2}\leq A_{1}\leq 1.$
Then%
\begin{equation*}
\frac{1+A_{2}z}{1+B_{2}z}\prec \frac{1+A_{1}z}{1+B_{1}z}.
\end{equation*}
\end{lemma}

\begin{lemma}
\label{7}\cite{miler} Let the function $g(z)$ be analytic and univalent in $%
E $ and let the functions $\theta (w)$ and $\varphi (w)$ be analytic in a
domain $D$ containing $g(E)$, with $\theta (w)\neq 0\ (w\in g(E))$. Set

$Q(z)=zg^{\prime }(z)\varphi (g(z))$ and $h(z)=\theta (g(z))+Q(z)$ and
suppose that

(i) $Q(z)$is univalently starlike in $E$

(ii) $\func{Re}\frac{zh^{\prime }(z)}{Q(z)}=\func{Re}\left\{ \frac{\theta
^{\prime }(g(z))}{\varphi (g(z))}+\frac{zQ^{\prime }(z)}{Q(z)}\right\} >0$ $%
(z\in E).\ $If $q(z)\ $is analytic in$\ E\ $with$\ q(0)=g(0),q(E)\subset D\ $%
and 
\begin{equation*}
\theta (q(z))+zq^{\prime }(z)\varphi (q(z))\prec \theta (g(z))+zg^{\prime
}(z)\varphi (g(z))=h(z)\ (z\in E),
\end{equation*}

then $q(z)\prec g(z)$ $(z\in E)$ and $g(z)$ is the best dominant.
\end{lemma}

\section{Main results}

\begin{theorem}
\label{5}Let $f\in N_{k,c}^{\alpha }\left( \lambda ,\mu ,\phi \right) .$
Then for $Re\frac{\mu k}{\lambda }\geq 0,$%
\begin{equation*}
e^{i\alpha }\left( \frac{z}{S_{k+1}^{c}f\left( z\right) }\right) ^{\mu
}\prec \frac{\mu k}{\lambda }\cos \alpha z^{-\frac{\mu k}{\lambda }%
}\dint\limits_{0}^{z}\phi \left( t\right) t^{\frac{\mu k}{\lambda }%
-1}dt+i\sin \alpha \prec \left( \cos \alpha \right) \phi \left( z\right)
+i\sin \alpha .
\end{equation*}%
This result is the best possible.
\end{theorem}

\begin{proof}
Consider%
\begin{equation}
p(z)=\frac{1}{\cos \alpha }\left\{ e^{i\alpha }\left( \frac{z}{%
S_{k+1}^{c}f\left( z\right) }\right) ^{\mu }-i\sin \alpha \right\} .
\label{j}
\end{equation}%
Then $p$ is analytic in $E$ with $p(0)=1$. Therefore$,\ $we have%
\begin{equation*}
e^{i\alpha }\left( \frac{z}{S_{k+1}^{c}f\left( z\right) }\right) ^{\mu
}=\left( \cos \alpha \right) p(z)+i\sin \alpha \text{.}
\end{equation*}%
Differentiating both sides and using $\left( \ref{h}\right) $ and
simplifying, we obtain%
\begin{equation*}
\frac{\lambda \left( \cos \alpha \right) zp^{\prime }\left( z\right) }{\mu k}%
=\lambda e^{i\alpha }\left\{ \left( \frac{z}{S_{k+1}^{c}f\left( z\right) }%
\right) ^{\mu }-\frac{S_{k}^{c}f\left( z\right) }{S_{k+1}^{c}f\left(
z\right) }\left( \frac{z}{S_{k+1}^{c}f\left( z\right) }\right) ^{\mu
}\right\} .
\end{equation*}%
It follows from above equation and $\left( \ref{j}\right) $ that%
\begin{eqnarray*}
&&p\left( z\right) +\frac{\lambda }{\mu k}zp^{\prime }\left( z\right) \\
&=&\frac{1}{\cos \alpha }\left[ e^{i\alpha }\left\{ \left( 1+\lambda \right)
\left( \frac{z}{S_{k+1}^{c}f\left( z\right) }\right) ^{\mu }-\lambda \frac{%
S_{k}^{c}f\left( z\right) }{S_{k+1}^{c}f\left( z\right) }\left( \frac{z}{%
S_{k+1}^{c}f\left( z\right) }\right) ^{\mu }\right\} -i\sin \alpha \right] .
\end{eqnarray*}%
Since $f\in N_{k,c}^{\alpha }\left( \lambda ,\mu ,\phi \right) ,$ therefore%
\begin{equation*}
p\left( z\right) +\frac{\lambda }{\mu k}zp^{\prime }\left( z\right) \prec
\phi \left( z\right) .
\end{equation*}%
Now using Lemma \ref{4} for $\beta =\frac{\mu k}{\lambda }$ with $Re\frac{%
\mu k}{\lambda }\geq 0,$ we obtain the required result.
\end{proof}

\begin{corollary}
Let $f\in N_{k,c}^{\alpha }\left( \lambda ,\mu ,\frac{1+Az}{1+Bz}\right) .$
Then for $k,\lambda \in 
%TCIMACRO{\U{211d} }%
%BeginExpansion
\mathbb{R}
%EndExpansion
$ and $\frac{\mu k}{\lambda }\geq 0,$%
\begin{equation*}
e^{i\alpha }\left( \frac{z}{S_{k+1}^{c}f\left( z\right) }\right) ^{\mu
}\prec h\left( z\right) \cos \alpha +i\sin \alpha ,
\end{equation*}%
where%
\begin{equation*}
h\left( z\right) =\left\{ 
\begin{array}{c}
\frac{A}{B}+\left( 1-\frac{A}{B}\right) \left( 1+Bz\right) ^{-1}\
_{2}F_{1}\left( 1,1,\frac{\mu k}{\lambda }+1;\frac{Bz}{1+Bz}\right) ,\ \ \ \
B\neq 0, \\ 
1+\frac{\mu k}{\mu k+\lambda }Az,\ \ \ \ \ \ \ \ \ \ \ \ \ \ \ \ \ \ \ \ \ \
\ \ \ \ \ \ \ \ \ \ \ \ \ \ \ \ \ \ \ \ \ \ \ \ B=0.%
\end{array}%
\right.
\end{equation*}%
Further%
\begin{equation*}
\func{Re}\left[ e^{i\alpha }\left( \frac{z}{S_{k+1}^{c}f\left( z\right) }%
\right) ^{\mu }\right] >\left( \cos \alpha \right) h\left( -1\right) .
\end{equation*}
\end{corollary}

\begin{proof}
Since $f\in N_{k,c}^{\alpha }\left( \lambda ,\mu ,\frac{1+Az}{1+Bz}\right) ,$
therefore from Theorem \ref{5}, we have%
\begin{equation}
e^{i\alpha }\left( \frac{z}{S_{k+1}^{c}f\left( z\right) }\right) ^{\mu
}\prec \frac{\mu k}{\lambda }\left( \cos \alpha \right) z^{-\frac{\mu k}{%
\lambda }}\dint\limits_{0}^{z}\frac{1+At}{1+Bt}t^{\frac{\mu k}{\lambda }%
-1}dt+i\sin \alpha .  \label{k}
\end{equation}%
Putting $t=zu$ and after simple calculations$,$ one can get%
\begin{equation*}
e^{i\alpha }\left( \frac{z}{S_{k+1}^{c}f\left( z\right) }\right) ^{\mu
}\prec \left\{ \frac{A}{B}+\frac{\mu k}{\lambda }\left( 1-\frac{A}{B}\right)
\dint\limits_{0}^{1}\left( 1+Buz\right) ^{-1}u^{\frac{\mu k}{\lambda }%
-1}dt\right\} \cos \alpha +i\sin \alpha .
\end{equation*}%
Now using Lemma \ref{1} for $a=1,\ b=\frac{\mu k}{\lambda },\ c=b+1$ and $%
B\neq 0,$ we obtain%
\begin{eqnarray*}
&&e^{i\alpha }\left( \frac{z}{S_{k+1}^{c}f\left( z\right) }\right) ^{\mu } \\
&\prec &\left( \frac{A}{B}+\left( 1-\frac{A}{B}\right) \left( 1+Bz\right)
^{-1}\ _{2}F_{1}\left( 1,1,\frac{\mu k}{\lambda }+1;\frac{Bz}{1+Bz}\right)
\right) \cos \alpha +i\sin \alpha .
\end{eqnarray*}%
For the case $B=0.$ It can easily be followed from $\left( \ref{k}\right) $
that%
\begin{eqnarray*}
e^{i\alpha }\left( \frac{z}{S_{k+1}^{c}f\left( z\right) }\right) ^{\mu }
&\prec &\left( \frac{\mu k}{\lambda }\dint\limits_{0}^{1}\left( 1+Atz\right)
t^{\frac{\mu k}{\lambda }-1}dt\right) \cos \alpha +i\sin \alpha . \\
&=&\frac{\mu k}{\lambda }\left\{ \left( \dint\limits_{0}^{1}t^{\frac{\mu k}{%
\lambda }-1}dt\right) +\dint\limits_{0}^{1}Azt^{\frac{\mu k}{\lambda }%
}dt\right\} \cos \alpha +i\sin \alpha . \\
&=&\left\{ 1+\frac{\mu k}{\mu k+\lambda }Az\right\} \cos \alpha +i\sin
\alpha .
\end{eqnarray*}%
Now we have to prove that $\func{Re}\left[ e^{i\alpha }\left( \frac{z}{%
S_{k+1}^{c}f\left( z\right) }\right) ^{\mu }\right] >\left( \cos \alpha
\right) h\left( -1\right) .$ From $\left( \ref{k}\right) ,$ we can have this
relation by using subordination%
\begin{equation*}
\frac{1}{\cos \alpha }\left\{ e^{i\alpha }\left( \frac{z}{S_{k+1}^{c}f\left(
z\right) }\right) ^{\mu }-i\sin \alpha \right\} =h\left( w\left( z\right)
\right) ,
\end{equation*}%
where $h\left( z\right) =\frac{\mu k}{\lambda }z^{-\frac{\mu k}{\lambda }%
}\dint\limits_{0}^{z}\frac{1+At}{1+Bt}t^{\frac{\mu k}{\lambda }-1}dt.\ $%
Therefore%
\begin{eqnarray*}
\func{Re}\left[ \frac{1}{\cos \alpha }\left\{ e^{i\alpha }\left( \frac{z}{%
S_{k+1}^{c}f\left( z\right) }\right) ^{\mu }\right\} \right] &=&\func{Re}%
\frac{\mu k}{\lambda }\dint\limits_{0}^{1}\frac{1+Atw\left( z\right) }{%
1+Btw\left( z\right) }t^{\frac{\mu k}{\lambda }-1}dt \\
&>&\frac{\mu k}{\lambda }\dint\limits_{0}^{1}\frac{1-At}{1-Bt}t^{\frac{\mu k%
}{\lambda }-1}dt \\
&=&h\left( -1\right) .
\end{eqnarray*}%
To show that this result is sharp, we have to prove that $\underset{%
\left\vert z\right\vert <1}{\inf }\left\{ \func{Re}h\left( z\right) \right\}
=h\left( -1\right) .$ Now%
\begin{equation*}
\func{Re}h\left( z\right) \geq \frac{\mu k}{\lambda }\dint\limits_{0}^{1}t^{%
\frac{\mu k}{\lambda }-1}\frac{1-Atr}{1-Btr}dt=h\left( -r\right) .
\end{equation*}%
Therefore $h\left( -r\right) \rightarrow h\left( -1\right) $ as $%
r\rightarrow 1^{-}.$
\end{proof}

\begin{theorem}
Let $e^{i\alpha }\left( \frac{z}{S_{k+1}^{c}f\left( z\right) }\right) ^{\mu
}\prec \phi \left( z\right) \cos \alpha +i\sin \alpha $ with $\phi \left(
z\right) =\frac{1+z}{1-z}.$ Then $f\in N_{k,c}^{\alpha }\left( \lambda ,\mu
,\phi \left( z\right) \right) $ for $\left\vert z\right\vert =r<-c+\sqrt{%
c^{2}+1},\ $where $c=\left\vert \frac{\lambda }{\mu k}\right\vert .$
\end{theorem}

\begin{proof}
\textbf{\ }Let 
\begin{equation*}
e^{i\alpha }\left( \frac{z}{S_{k+1}^{c}f\left( z\right) }\right) ^{\mu
}=p\left( z\right) \cos \alpha +i\sin \alpha ,
\end{equation*}%
where $p\left( z\right) \prec \frac{1+z}{1-z}.$ Then from Theorem \ref{5},
we have%
\begin{eqnarray*}
&&p\left( z\right) +\frac{\lambda }{\mu k}zp^{\prime }\left( z\right) \\
&=&\frac{1}{\cos \alpha }\left[ e^{i\alpha }\left\{ \left( 1+\lambda \right)
\left( \frac{z}{S_{k+1}^{c}f\left( z\right) }\right) ^{\mu }-\lambda \frac{%
S_{k}^{c}f\left( z\right) }{S_{k+1}^{c}f\left( z\right) }\left( \frac{z}{%
S_{k+1}^{c}f\left( z\right) }\right) ^{\mu }\right\} -i\sin \alpha \right] .
\end{eqnarray*}%
Since $p\left( z\right) \prec \frac{1+z}{1-z},$ then it is well known that
see\ \cite{goo} 
\begin{equation}
\frac{1-r}{1+r}\leq \func{Re}p\left( z\right) \leq \left\vert p\left(
z\right) \right\vert \leq \frac{1+r}{1-r}\text{ and }\left\vert zp^{\prime
}\left( z\right) \right\vert \leq \frac{2r\func{Re}p\left( z\right) }{1-r^{2}%
}.  \label{m}
\end{equation}%
Thus, we have%
\begin{eqnarray*}
&&\func{Re}\frac{1}{\cos \alpha }\left[ e^{i\alpha }\left\{ \left( 1+\lambda
\right) \left( \frac{z}{S_{k+1}^{c}f\left( z\right) }\right) ^{\mu }-\lambda 
\frac{S_{k}^{c}f\left( z\right) }{S_{k+1}^{c}f\left( z\right) }\left( \frac{z%
}{S_{k+1}^{c}f\left( z\right) }\right) ^{\mu }\right\} -i\sin \alpha \right]
\\
&\geq &\func{Re}p\left( z\right) -\left\vert \frac{\lambda }{\mu k}%
\right\vert \left\vert zp^{\prime }\left( z\right) \right\vert
\end{eqnarray*}%
Using $\left( \ref{m}\right) ,$ we obtain%
\begin{eqnarray*}
&&\func{Re}\frac{1}{\cos \alpha }\left[ e^{i\alpha }\left\{ \left( 1+\lambda
\right) \left( \frac{z}{S_{k+1}^{c}f\left( z\right) }\right) ^{\mu }-\lambda 
\frac{S_{k}^{c}f\left( z\right) }{S_{k+1}^{c}f\left( z\right) }\left( \frac{z%
}{S_{k+1}^{c}f\left( z\right) }\right) ^{\mu }\right\} -i\sin \alpha \right]
\\
&\geq &\func{Re}p\left( z\right) -\frac{2cr\func{Re}p\left( z\right) }{%
1-r^{2}} \\
&=&\func{Re}p\left( z\right) \frac{1-r^{2}-2cr}{1-r^{2}}.
\end{eqnarray*}%
Since $p\left( z\right) \prec \frac{1+z}{1-z},$ therefore $\func{Re}p\left(
z\right) >0.$ This implies that $f\in N_{k,c}^{\alpha }\left( \lambda ,\mu
,\phi \left( z\right) \right) $ for $r<-c+\sqrt{c^{2}+1}.$ This result is
sharp for the function $p\left( z\right) =\frac{1+z}{1-z}.$
\end{proof}

\begin{theorem}
Let $0<\mu <1,\ k=p+\left( b+2\right) /2\neq 0,-1,-2,\ldots ,b,c,p\in 
%TCIMACRO{\U{2102} }%
%BeginExpansion
\mathbb{C}
%EndExpansion
$. Then%
\begin{equation*}
N_{k,c}^{0}\left( \lambda _{2},\mu ,\phi \right) \subset N_{k,c}^{0}\left(
\lambda _{1},\mu ,\phi \right) ,\text{ }0\leq \lambda _{1}<\lambda _{2}.
\end{equation*}
\end{theorem}

\begin{proof}
Since $f\in N_{k,c}^{\alpha }\left( \lambda _{2},\mu ,\phi \right) ,$
therefore we have%
\begin{equation*}
h_{1}\left( z\right) =\left( 1+\lambda _{2}\right) \left( \frac{z}{%
S_{k+1}^{c}f\left( z\right) }\right) ^{\mu }-\lambda _{2}\frac{%
S_{k}^{c}f\left( z\right) }{S_{k+1}^{c}f\left( z\right) }\left( \frac{z}{%
S_{k+1}^{c}f\left( z\right) }\right) ^{\mu }\prec \phi (z).
\end{equation*}%
From Theorem \ref{5} for $\alpha =0$, we write%
\begin{equation*}
h_{2}\left( z\right) =\left( \frac{z}{S_{k+1}^{c}f\left( z\right) }\right)
^{\mu }\prec \phi (z),\text{ }z\in E\text{.}
\end{equation*}%
Now for $\lambda _{1}\geq 0$, we obtain%
\begin{eqnarray*}
&&\left( 1+\lambda _{1}\right) \left( \frac{z}{S_{k+1}^{c}f\left( z\right) }%
\right) ^{\mu }-\lambda _{1}\frac{S_{k}^{c}f\left( z\right) }{%
S_{k+1}^{c}f\left( z\right) }\left( \frac{z}{S_{k+1}^{c}f\left( z\right) }%
\right) ^{\mu } \\
&=&(1-\frac{\lambda _{1}}{\lambda _{2}})\left( \frac{z}{S_{k+1}^{c}f\left(
z\right) }\right) ^{\mu }+ \\
&&\frac{\lambda _{1}}{\lambda _{2}}\left\{ \left( 1+\lambda _{2}\right)
\left( \frac{z}{S_{k+1}^{c}f\left( z\right) }\right) ^{\mu }-\lambda _{2}%
\frac{S_{k}^{c}f\left( z\right) }{S_{k+1}^{c}f\left( z\right) }\left( \frac{z%
}{S_{k+1}^{c}f\left( z\right) }\right) ^{\mu }\right\} \\
&=&\frac{\lambda _{1}}{\lambda _{2}}h_{1}(z)+(1-\frac{\lambda _{1}}{\lambda
_{2}})h_{2}(z).
\end{eqnarray*}%
Using the convexity of the class of the functions $\phi (z)$ and Lemma \ref%
{3}, we write%
\begin{equation*}
\frac{\lambda _{1}}{\lambda _{2}}h_{1}(z)+(1-\frac{\lambda _{1}}{\lambda _{2}%
})h_{2}(z)\prec \phi (z)\text{, \ }z\in E\text{,}
\end{equation*}%
This implies that $f\in N_{k,c}^{0}\left( \lambda _{1},\mu ,\phi \right) $.
Hence the proof of the theorem is complete.
\end{proof}

\begin{corollary}
Let $0<\mu <1,\ k=p+\left( b+2\right) /2\neq 0,-1,-2,\ldots ,b,c,p\in 
%TCIMACRO{\U{2102} }%
%BeginExpansion
\mathbb{C}
%EndExpansion
$. Then for $-1\leq B_{1}\leq B_{2}<A_{2}\leq A_{1}\leq 1,$%
\begin{equation*}
N_{k,c}^{0}\left( \lambda _{2},\mu ,\frac{1+A_{2}z}{1+B_{2}z}\right) \subset
N_{k,c}^{0}\left( \lambda _{1},\mu ,\frac{1+A_{1}z}{1+B1z}\right) ,\text{ }%
0\leq \lambda _{1}<\lambda _{2}\text{, }z\in E\text{.}
\end{equation*}
\end{corollary}

\begin{proof}
Let $f\in N_{k,c}^{0}\left( \lambda _{2},\mu ,\frac{1+A_{2}z}{1+B_{2}z}%
\right) .$ Then%
\begin{equation*}
h_{1}\left( z\right) =\left( 1+\lambda _{2}\right) \left( \frac{z}{%
S_{k+1}^{c}f\left( z\right) }\right) ^{\mu }-\lambda _{2}\frac{%
S_{k}^{c}f\left( z\right) }{S_{k+1}^{c}f\left( z\right) }\left( \frac{z}{%
S_{k+1}^{c}f\left( z\right) }\right) ^{\mu }\prec \frac{1+A_{2}z}{1+B_{2}z}.
\end{equation*}%
Since $-1\leq B_{1}\leq B_{2}<A_{2}\leq A_{1}\leq 1,$ therefore by Lemma \ref%
{6}, we have%
\begin{equation*}
h_{1}\left( z\right) =\left( 1+\lambda _{2}\right) \left( \frac{z}{%
S_{k+1}^{c}f\left( z\right) }\right) ^{\mu }-\lambda _{2}\frac{%
S_{k}^{c}f\left( z\right) }{S_{k+1}^{c}f\left( z\right) }\left( \frac{z}{%
S_{k+1}^{c}f\left( z\right) }\right) ^{\mu }\prec \frac{1+A_{1}z}{1+B_{1}z}.
\end{equation*}%
Theorem \ref{5} implies for $\phi (z)=\frac{1+A_{1}z}{1+B_{1}z}$ that%
\begin{equation*}
h_{2}\left( z\right) =\left( \frac{z}{S_{k+1}^{c}f\left( z\right) }\right)
^{\mu }\prec \frac{1+A_{1}z}{1+B_{1}z}.
\end{equation*}%
Now for $\lambda _{2}>\lambda _{1}\geq 0,$%
\begin{eqnarray*}
&&\left( 1+\lambda _{1}\right) \left( \frac{z}{S_{k+1}^{c}f\left( z\right) }%
\right) ^{\mu }-\lambda _{1}\frac{S_{k}^{c}f\left( z\right) }{%
S_{k+1}^{c}f\left( z\right) }\left( \frac{z}{S_{k+1}^{c}f\left( z\right) }%
\right) ^{\mu } \\
&=&(1-\frac{\lambda _{1}}{\lambda _{2}})\left( \frac{z}{S_{k+1}^{c}f\left(
z\right) }\right) ^{\mu }+ \\
&&\frac{\lambda _{1}}{\lambda _{2}}\left\{ \left( 1+\lambda _{2}\right)
\left( \frac{z}{S_{k+1}^{c}f\left( z\right) }\right) ^{\mu }-\lambda _{2}%
\frac{S_{k}^{c}f\left( z\right) }{S_{k+1}^{c}f\left( z\right) }\left( \frac{z%
}{S_{k+1}^{c}f\left( z\right) }\right) ^{\mu }\right\} \\
&=&\frac{\lambda _{1}}{\lambda _{2}}h_{1}(z)+(1-\frac{\lambda _{1}}{\lambda
_{2}})h_{2}(z).
\end{eqnarray*}%
Using the convexity of the function $\frac{1+A_{1}z}{1+B_{1}z}$ with Lemma %
\ref{3}, we write%
\begin{equation*}
\frac{\lambda _{1}}{\lambda _{2}}h_{1}(z)+(1-\frac{\lambda _{1}}{\lambda _{2}%
})h_{2}(z)\prec \frac{1+A_{1}z}{1+B_{1}z}\text{, \ }z\in E\text{,}
\end{equation*}%
This implies that $f\in N_{k,c}^{0}\left( \lambda _{1},\mu ,\frac{1+A_{1}z}{%
1+B_{1}z}\right) $.
\end{proof}

\begin{theorem}
Let $f\in N_{k,c}^{0}\left( \lambda ,\mu ,\phi \right) ,$ $0<\mu <1,\
k=p+\left( b+2\right) /2\neq 0,-1,-2,\ldots ,b,c,p\in 
%TCIMACRO{\U{2102} }%
%BeginExpansion
\mathbb{C}
%EndExpansion
$ and $\lambda \leq -1.$ Then%
\begin{equation*}
\frac{S_{k}^{c}f\left( z\right) }{S_{k+1}^{c}f\left( z\right) }\left( \frac{z%
}{S_{k+1}^{c}f\left( z\right) }\right) ^{\mu }\prec \phi (z).
\end{equation*}
\end{theorem}

\begin{proof}
Since $f\in N_{k,c}^{0}\left( \lambda ,\mu ,\phi \right) ,\ $therefore we
have%
\begin{equation*}
\left( 1+\lambda \right) \left( \frac{z}{S_{k+1}^{c}f\left( z\right) }%
\right) ^{\mu }-\lambda \frac{S_{k}^{c}f\left( z\right) }{S_{k+1}^{c}f\left(
z\right) }\left( \frac{z}{S_{k+1}^{c}f\left( z\right) }\right) ^{\mu }\prec
\phi (z).
\end{equation*}%
Now consider%
\begin{eqnarray*}
\lambda \frac{S_{k}^{c}f\left( z\right) }{S_{k+1}^{c}f\left( z\right) }%
\left( \frac{z}{S_{k+1}^{c}f\left( z\right) }\right) ^{\mu } &=&\left(
1+\lambda \right) \left( \frac{z}{S_{k+1}^{c}f\left( z\right) }\right) ^{\mu
}+\lambda \frac{S_{k}^{c}f\left( z\right) }{S_{k+1}^{c}f\left( z\right) }%
\left( \frac{z}{S_{k+1}^{c}f\left( z\right) }\right) ^{\mu } \\
&&-\left( 1+\lambda \right) \left( \frac{z}{S_{k+1}^{c}f\left( z\right) }%
\right) ^{\mu }.
\end{eqnarray*}%
This implies that%
\begin{eqnarray*}
\frac{S_{k}^{c}f\left( z\right) }{S_{k+1}^{c}f\left( z\right) }\left( \frac{z%
}{S_{k+1}^{c}f\left( z\right) }\right) ^{\mu } &=&\left( 1+\frac{1}{\lambda }%
\right) \left( \frac{z}{S_{k+1}^{c}f\left( z\right) }\right) ^{\mu } \\
&&-\frac{1}{\lambda }\left\{ \left( 1+\lambda \right) \left( \frac{z}{%
S_{k+1}^{c}f\left( z\right) }\right) ^{\mu }+\lambda \frac{S_{k}^{c}f\left(
z\right) }{S_{k+1}^{c}f\left( z\right) }\left( \frac{z}{S_{k+1}^{c}f\left(
z\right) }\right) ^{\mu }\right\}
\end{eqnarray*}%
Using Theorem \ref{5}, Lemma \ref{3} and the convexity of $\phi (z)$ with $%
\lambda \leq -1,$ we have the required result.
\end{proof}

\begin{theorem}
Let $f\in N_{k,c}^{\alpha }\left( \lambda ,\mu ,h\right) ,h(z)=\frac{1+Az}{%
1+Bz}+\frac{\lambda \mu }{k}\frac{\left( A-B\right) z}{\left( 1+Bz\right)
^{2}}.$ Then for $\func{Re}\frac{\lambda }{\mu k}>0,$%
\begin{equation*}
e^{i\alpha }\left( \frac{z}{S_{k+1}^{c}f\left( z\right) }\right) ^{\mu
}\prec \left( \cos \alpha \right) \phi \left( z\right) +i\sin \alpha \text{.}
\end{equation*}%
where $\phi \left( z\right) =\frac{1+Az}{1+Bz}.$ This result is \ the best
possible.
\end{theorem}

\begin{proof}
Consider%
\begin{equation*}
p(z)=\frac{1}{\cos \alpha }\left\{ e^{i\alpha }\left( \frac{z}{%
S_{k+1}^{c}f\left( z\right) }\right) ^{\mu }-i\sin \alpha \right\} .
\end{equation*}%
Then $p$ is analytic in $E$ with $p(0)=1$. Therefore$,\ $we have%
\begin{equation*}
e^{i\alpha }\left( \frac{z}{S_{k+1}^{c}f\left( z\right) }\right) ^{\mu
}=\left( \cos \alpha \right) p(z)+i\sin \alpha \text{.}
\end{equation*}%
Differentiating both sides, using $\left( \ref{h}\right) $ and simplifying,
we obtain%
\begin{equation*}
\frac{\lambda \left( \cos \alpha \right) zp^{\prime }\left( z\right) }{\mu k}%
=\lambda e^{i\alpha }\left\{ \left( \frac{z}{S_{k+1}^{c}f\left( z\right) }%
\right) ^{\mu }-\frac{S_{k}^{c}f\left( z\right) }{S_{k+1}^{c}f\left(
z\right) }\left( \frac{z}{S_{k+1}^{c}f\left( z\right) }\right) ^{\mu
}\right\} .
\end{equation*}%
It follows from above equation and $\left( \ref{j}\right) $ that%
\begin{eqnarray*}
&&p\left( z\right) +\frac{\lambda }{\mu k}zp^{\prime }\left( z\right) \\
&=&\frac{1}{\cos \alpha }\left[ e^{i\alpha }\left\{ \left( 1+\lambda \right)
\left( \frac{z}{S_{k+1}^{c}f\left( z\right) }\right) ^{\mu }-\lambda \frac{%
S_{k}^{c}f\left( z\right) }{S_{k+1}^{c}f\left( z\right) }\left( \frac{z}{%
S_{k+1}^{c}f\left( z\right) }\right) ^{\mu }\right\} -i\sin \alpha \right] .
\end{eqnarray*}%
Since $f\in N_{k,c}^{\alpha }\left( \lambda ,\mu ,h\right) ,$ therefore%
\begin{equation*}
p\left( z\right) +\frac{\lambda }{\mu k}zp^{\prime }\left( z\right) \prec
h\left( z\right) .
\end{equation*}%
Now we choose $g\left( z\right) =\frac{1+Az}{1+Bz},$ then $\theta (w)=w$ and 
$\varphi (w)=\frac{\mu k}{\lambda }.$ It is clear that $g\left( z\right) $
is analytic in $E$ with $g\left( 0\right) =1.$ Also $\theta (w)$ and $%
\varphi (w)$ are analytic with $\theta (w)\neq 0$.

We see that%
\begin{equation}
Q\left( z\right) =zg^{\prime }(z)\varphi (g(z))=\frac{\mu k}{\lambda }\frac{%
\left( A-B\right) z}{\left( 1+Bz\right) ^{2}}.  \label{n}
\end{equation}%
We have to prove that $Q\left( z\right) $ is starlike. In other words we
show that $\func{Re}\frac{zQ^{\prime }(z)}{Q(z)}>0.$ From $\left( \ref{n}%
\right) $, we have%
\begin{eqnarray*}
\func{Re}\frac{zQ^{\prime }(z)}{Q(z)} &=&\func{Re}\left\{ 1-\frac{2Bz}{1+Bz}%
\right\} \\
&=&1-2B\func{Re}\frac{re^{i\psi }}{1+Bre^{i\psi }}\ \ \ \ \ \ \ \ \ \left(
z=re^{i\psi }\right) \\
&=&\frac{1-B^{2}r^{2}}{\left( 1+Br\cos \psi \right) ^{2}+B^{2}r^{2}\sin
^{2}\psi }.
\end{eqnarray*}%
Since $-1\leq B<1,$ $r<1.$ This implies that $\func{Re}\frac{zQ^{\prime }(z)%
}{Q(z)}>0.$ Consider%
\begin{eqnarray*}
\func{Re}\frac{zh^{\prime }(z)}{Q(z)} &=&\func{Re}\left\{ \frac{\theta
^{\prime }(g(z))}{\varphi (g(z))}+\frac{zQ^{\prime }(z)}{Q(z)}\right\} \\
&=&\func{Re}\frac{\lambda }{\mu k}+\func{Re}\frac{zQ^{\prime }(z)}{Q(z)}>0.
\end{eqnarray*}%
Using Lemma \ref{7}, we have $e^{i\alpha }\left( \frac{z}{S_{k+1}^{c}f\left(
z\right) }\right) ^{\mu }\prec \left( \cos \alpha \right) \phi \left(
z\right) +i\sin \alpha $. The function $\phi \left( z\right) =\frac{1+Az}{%
1+Bz}$ is the best possible.
\end{proof}

\begin{theorem}
Let $f\in N_{k,c}^{\alpha }\left( \lambda ,\mu ,\frac{1+Az}{1+Bz}\right) .$
Then for $k,\lambda \in 
%TCIMACRO{\U{211d} }%
%BeginExpansion
\mathbb{R}
%EndExpansion
$ and $\frac{\mu k}{\lambda }\geq 0,$%
\begin{eqnarray*}
&&\left. 
\begin{array}{c}
\frac{A}{B}+\left( 1-\frac{A}{B}\right) \ _{2}F_{1}\left( 1,\frac{\mu k}{%
\lambda },\frac{\mu k}{\lambda }+1;B\right) ,\ B\neq 0, \\ 
1-\frac{\mu k}{\mu k+\lambda }A,\ \ \ \ \ \ \ \ \ \ \ \ \ \ \ \ \ \ \ \ \ \
\ \ \ \ \ \ \ \ B=0.%
\end{array}%
\right\} \\
&<&\frac{1}{\cos \alpha }\func{Re}\left\{ e^{i\alpha }\left( \frac{z}{%
S_{k+1}^{c}f\left( z\right) }\right) ^{\mu }\right\} <\left\{ 
\begin{array}{c}
\frac{A}{B}+\left( 1-\frac{A}{B}\right) \ _{2}F_{1}\left( 1,\frac{\mu k}{%
\lambda },\frac{\mu k}{\lambda }+1;-B\right) ,\ B\neq 0, \\ 
1+\frac{\mu k}{\mu k+\lambda }A,\ \ \ \ \ \ \ \ \ \ \ \ \ \ \ \ \ \ \ \ \ \
\ \ \ \ \ \ \ \ \ \ B=0.%
\end{array}%
\right.
\end{eqnarray*}
\end{theorem}

\begin{proof}
Since $f\in N_{k,c}^{\alpha }\left( \lambda ,\mu ,\frac{1+Az}{1+Bz}\right) ,$
therefore by using $\left( \ref{k}\right) ,$ we have%
\begin{equation*}
\frac{1}{\cos \alpha }\func{Re}\left\{ e^{i\alpha }\left( \frac{z}{%
S_{k+1}^{c}f\left( z\right) }\right) ^{\mu }\right\} \prec \func{Re}\frac{%
\mu k}{\lambda }\dint\limits_{0}^{1}\frac{1+Atz}{1+Btz}t^{\frac{\mu k}{%
\lambda }-1}dt.
\end{equation*}%
It follows from the definition of subordination that%
\begin{eqnarray*}
\frac{1}{\cos \alpha }\func{Re}\left\{ e^{i\alpha }\left( \frac{z}{%
S_{k+1}^{c}f\left( z\right) }\right) ^{\mu }\right\} &<&\underset{\left\vert
z\right\vert <1}{\sup }\func{Re}\left\{ \frac{\mu k}{\lambda }%
\dint\limits_{0}^{1}\frac{1+Atz}{1+Btz}t^{\frac{\mu k}{\lambda }-1}dt\right\}
\\
&\leq &\left\{ \frac{\mu k}{\lambda }\dint\limits_{0}^{1}\underset{%
\left\vert z\right\vert <1}{\sup }\func{Re}\left\{ \frac{1+Atz}{1+Btz}%
\right\} t^{\frac{\mu k}{\lambda }-1}dt\right\} \\
&<&\frac{\mu k}{\lambda }\dint\limits_{0}^{1}\frac{1+At}{1+Bt}t^{\frac{\mu k%
}{\lambda }-1}dt \\
&=&\frac{\mu k}{\lambda }\dint\limits_{0}^{1}\left\{ A/B+\left( \frac{1-A/B}{%
1+Bt}\right) \right\} t^{\frac{\mu k}{\lambda }-1}dt.
\end{eqnarray*}%
Now using Lemma \ref{1} for the case $B\neq 0,$ we have%
\begin{equation*}
\frac{1}{\cos \alpha }\func{Re}\left\{ e^{i\alpha }\left( \frac{z}{%
S_{k+1}^{c}f\left( z\right) }\right) ^{\mu }\right\} <\frac{A}{B}+\left( 1-%
\frac{A}{B}\right) \text{ }_{2}F_{1}\left( 1,\frac{\mu k}{\lambda },\frac{%
\mu k}{\lambda }+1;-B\right) .
\end{equation*}%
When $B=0,$ it can be easily seen that%
\begin{eqnarray*}
\frac{1}{\cos \alpha }\func{Re}\left\{ e^{i\alpha }\left( \frac{z}{%
S_{k+1}^{c}f\left( z\right) }\right) ^{\mu }\right\} &<&\frac{\mu k}{\lambda 
}\dint\limits_{0}^{1}\left( 1+At\right) t^{\frac{\mu k}{\lambda }-1}dt \\
&=&1+\frac{\mu k}{\mu k+\lambda }A.
\end{eqnarray*}%
We also have%
\begin{eqnarray*}
\frac{1}{\cos \alpha }\func{Re}\left\{ e^{i\alpha }\left( \frac{z}{%
S_{k+1}^{c}f\left( z\right) }\right) ^{\mu }\right\} &>&\underset{\left\vert
z\right\vert <1}{\inf }\func{Re}\left\{ \frac{\mu k}{\lambda }%
\dint\limits_{0}^{1}\frac{1+Atz}{1+Btz}t^{\frac{\mu k}{\lambda }-1}dt\right\}
\\
&\geq &\left\{ \frac{\mu k}{\lambda }\dint\limits_{0}^{1}\underset{%
\left\vert z\right\vert <1}{\inf }\func{Re}\left\{ \frac{1+Atz}{1+Btz}%
\right\} t^{\frac{\mu k}{\lambda }-1}dt\right\} \\
&>&\frac{\mu k}{\lambda }\dint\limits_{0}^{1}\frac{1-At}{1-Bt}t^{\frac{\mu k%
}{\lambda }-1}dt \\
&=&\frac{\mu k}{\lambda }\dint\limits_{0}^{1}\left\{ A/B+\left( \frac{1-A/B}{%
1-Bt}\right) \right\} t^{\frac{\mu k}{\lambda }-1}dt.
\end{eqnarray*}%
Using again Lemma \ref{1}, we have the required result.
\end{proof}

\begin{theorem}
Let $f\in N_{k,c}^{\alpha }\left( \lambda ,\mu ,\frac{1+Az}{1+Bz}\right) .$
Then for $k,\lambda \in 
%TCIMACRO{\U{211d} }%
%BeginExpansion
\mathbb{R}
%EndExpansion
$ and $\frac{\mu k}{\lambda }\geq 0,$%
\begin{eqnarray*}
&&\left. 
\begin{array}{c}
\frac{A}{B}+\left( 1-\frac{A}{B}\right) \ _{2}F_{1}\left( 1,\frac{\mu k}{%
\lambda },\frac{\mu k}{\lambda }+1;Br\right) ,\ B\neq 0, \\ 
1-\frac{\mu k}{\mu k+\lambda }A,\ \ \ \ \ \ \ \ \ \ \ \ \ \ \ \ \ \ \ \ \ \
\ \ \ \ \ \ \ \ \ \ \ B=0.%
\end{array}%
\right\} \\
&\leq &\left\vert \frac{1}{\cos \alpha }\left\{ e^{i\alpha }\left( \frac{z}{%
S_{k+1}^{c}f\left( z\right) }\right) ^{\mu }\right\} -i\sin \alpha
\right\vert \leq \left\{ 
\begin{array}{c}
\frac{A}{B}+\left( 1-\frac{A}{B}\right) \ _{2}F_{1}\left( 1,\frac{\mu k}{%
\lambda },\frac{\mu k}{\lambda }+1;-Br\right) ,\ B\neq 0, \\ 
1+\frac{\mu k}{\mu k+\lambda }A,\ \ \ \ \ \ \ \ \ \ \ \ \ \ \ \ \ \ \ \ \ \
\ \ \ \ \ \ \ \ \ \ B=0.%
\end{array}%
\right.
\end{eqnarray*}
\end{theorem}

\begin{proof}
Since $f\in N_{k,c}^{\alpha }\left( \lambda ,\mu ,\frac{1+Az}{1+Bz}\right) ,$
therefore by using $\left( \ref{k}\right) ,$ we have%
\begin{equation*}
\frac{1}{\cos \alpha }\left\{ e^{i\alpha }\left( \frac{z}{S_{k+1}^{c}f\left(
z\right) }\right) ^{\mu }-i\sin \alpha \right\} \prec \frac{\mu k}{\lambda }%
\dint\limits_{0}^{1}\frac{1+Atz}{1+Btz}t^{\frac{\mu k}{\lambda }-1}dt.
\end{equation*}%
It follows from the definition of subordination that%
\begin{equation*}
\frac{1}{\cos \alpha }\left\{ e^{i\alpha }\left( \frac{z}{S_{k+1}^{c}f\left(
z\right) }\right) ^{\mu }-i\sin \alpha \right\} =\left\{ \frac{\mu k}{%
\lambda }\dint\limits_{0}^{1}\frac{1+Atw\left( z\right) }{1+Btw\left(
z\right) }t^{\frac{\mu k}{\lambda }-1}dt\right\} ,
\end{equation*}%
where $w\left( z\right) =c_{1}z+c_{2}z^{2}+\ldots $ is analytic and $%
\left\vert w\left( z\right) \right\vert \leq \left\vert z\right\vert .$
Therefore%
\begin{equation*}
\left\vert \frac{1}{\cos \alpha }\left\{ e^{i\alpha }\left( \frac{z}{%
S_{k+1}^{c}f\left( z\right) }\right) ^{\mu }-i\sin \alpha \right\}
\right\vert \leq \left\{ \frac{\mu k}{\lambda }\dint\limits_{0}^{1}\frac{%
1+Atr}{1+Btr}t^{\frac{\mu k}{\lambda }-1}dt\right\} .
\end{equation*}%
Now using the same process as in the theorem above, we get the required
result.
\end{proof}

\end{document}